\documentclass[letterpaper, 10pt, conference]{ieeeconf}
\IEEEoverridecommandlockouts
\usepackage{graphicx}
\usepackage{psfrag}
\usepackage{amsmath,amssymb,comment}
\usepackage{multirow}
\usepackage{url}
\usepackage{color}

\newcommand{\eg}{{\it e.g., }}

\newcommand{\ie}{{\it i.e., }}

\newcommand{\E}[1]{\mathbf{E}[#1]}
\newcommand{\Estar}[1]{\mathbf{E}^{\text{lin}}[#1]}

\newtheorem{thm}{Theorem}
\newtheorem{rem}{Remark}
\newtheorem{prop}{Proposition}
\newtheorem{lem}{Lemma}
\newtheorem{cor}{Corollary}

\title{\LARGE \bf
Differentially Private State Estimation \\ in Distribution Networks with Smart Meters
}

\author{Henrik Sandberg, Gy\"{o}rgy D\'{a}n, and Ragnar Thobaben
\thanks{This work was supported in part by the Swedish Research Council (grant~2013-5523), the Swedish Foundation for Strategic Research through the ICT-Psi project, the EU-project SPARKS, and by the ACCESS Linnaeus Centre.}
\thanks{The authors are with the ACCESS Linnaeus Centre, School of Electrical Engineering, KTH Royal Institute of Technology, Stockholm, Sweden. Email: {\tt \{hsan,gyuri,ragnart\}@kth.se} }
}

\begin{document}
\maketitle
\thispagestyle{empty}
\pagestyle{empty}

\begin{abstract}
State estimation is routinely being performed in high-voltage power transmission grids in order to assist in operation and to detect faulty equipment. In low- and medium-voltage power distribution grids, on the other hand, few real-time measurements are traditionally available, and operation is often conducted based on predicted and historical data. Today, in many parts of the world, smart meters have been deployed at many customers, and their measurements could in principle be shared with the operators in real time to enable improved state estimation. However, customers may feel reluctance in doing so due to privacy concerns. We therefore propose state estimation schemes for a distribution grid model, which ensure differential privacy to the customers. In particular, the state estimation schemes optimize different performance criteria, and a trade-off between a lower bound on the estimation performance versus the customers' differential privacy is derived. The proposed framework is general enough to be applicable also to other distribution networks, such as water and gas networks.
\end{abstract}

\section{Introduction}
State estimation has long been an essential component of electric power system monitoring and control systems~\cite{Abur04}. In transmission systems it is typically based on near-real time measurements of voltages, power flows, and more recently, voltage phasors. In low- and medium-voltage distribution systems it is, however, often based on predicted loads in lack of a dedicated communication infrastructure and measurement devices~\cite{Kersting2006Book,AbdelMajeed+12}. Due to the proliferation of volatile distributed renewable generation (\eg residential PV systems), predicted loads will soon no longer be sufficient to achieve a good enough state estimate that would allow efficient and safe operation of the distribution grid, \eg voltage conservation and VAR~control (VVC), or fault location, isolation and restoration (FLIR), and real-time measurements will therefore become necessary.

While a dedicated real-time monitoring infrastructure for distribution grids may be too expensive,
automated and smart meters, which are being deployed world-wide for billing purposes, could be used for providing the operators frequent measurements of the loads, thereby enabling dynamic state estimation. Nonetheless, collecting frequent measurement data about residential or industrial customers' loads can be used to invade the customers' privacy in a variety of ways~\cite{Molina-Markham:2010}.

To protect customers' privacy but at the same time allow frequent smart meter measurements for system operation, a number of privacy-preserving aggregation schemes have been proposed recently~\cite{Bohli:2010,Li:2011,acscastelluccia}.
Aggregation can be done via a trusted third party~\cite{Bohli:2010} or by using cryptographic solutions such as homomorphic encryption~\cite{Li:2011}, but these solutions increase system complexity significantly.
An alternative aggregation scheme that does not increase system complexity relies on adding random noise, typically Gaussian or Laplacian,
to the measurement data submitted by the customers. Combined with careful protocol design, aggregation with random noise
has been shown to provide differential privacy~\cite{acscastelluccia}, a
probabilistic notion of privacy that is widely used in statistical databases~\cite{dwork+06,dwork+10} and has recently found application in filtering and control~\cite{nypappas,Huang+14}.
While adding noise may help preserve privacy, a fundamental question is how it would affect the quality of state estimation,
and ultimately the power system applications that rely on it.

In this paper, we address this question by providing a characterization of the trade-off between differential privacy
and the mean distortion of the state estimate in a single feeder of a
distribution network based on load measurements from smart meters.
We characterize the $\epsilon$- and the $(\epsilon,\delta)$-differential privacy under Laplacian and Gaussian
additive random noise, respectively. We express the maximum a~posteriori state estimate as the solution
of a convex programming problem, and provide a closed-form expression for the linear minimum mean square error estimate.
We show that a customer by adding random noise may only have to give up a little bit extra of privacy, and can still dramatically improve the
state estimation performance, which could enable optimizing financial incentives
for trading estimation performance and customer privacy.

The trade-off between privacy and estimation quality has been considered for inter-area state estimation in~\cite{Sankar2011sgc}, where the trade-off between the mutual information and the expected distortion was investigated in an information-theoretic framework. The authors in~\cite{Rajagopalan2013tsg} studied the trade-off between the distortion of the estimate of a random variable modeling an electric load and the information leakage quantified by the mutual information.
Instead of mutual information, in this paper we use differential privacy, which allows
us to evaluate privacy when different users' data are non-stationary, and even without a statistical model of the data.
In~\cite{nypappas}, the authors study the differential privacy of the inputs to a dynamical system when
releasing the outputs to a potential adversary, and design a differentially private Kalman filter
that protects the input variables. In this paper, the setup is somewhat similar but we do not consider
a dynamical system and we do analyze the case when the filter has access to side information, beyond the control of the inputs.
In~\cite{Huang+14}, the authors evaluate the cost of differential privacy in a system of agents that try to estimate their environment in a distributed manner. In our work, the estimation is done by a central entity instead of the agents themselves.

The rest of the paper is organized as follows. Section~\ref{sec:model} describes our model and the problem formulation. Section~\ref{sec:diffpriv} defines differential privacy in the context of state estimation. Section~\ref{sec:statest} provides optimal state estimation algorithms under various assumptions and bounds on their accuracy. Section~\ref{sec:tradeoff} analyzes the trade-off between privacy and estimation error, and Section~\ref{sec:conclusion} concludes the paper.

\section{Modeling and Problem Formulation}
\label{sec:model}
Operators of low- and medium-voltage power distribution networks today typically only have access to a relatively small number of  real-time measurements \cite{Kersting2006Book,AbdelMajeed+12}. It is then not possible to run standard state estimation algorithms as is commonly used in high-voltage transmission grids \cite{Abur04}. In this work, we therefore analyze how low-level measurements of the customers' loads can be integrated in state estimation schemes, while ensuring the customers' privacy.

We consider a single distribution line that is connected to a substation under the control of an operator, see Fig.~\ref{fig:powerline}. Since distribution grids are typically radial, this is not a severe restriction.
There are $N$ service drops or laterals along the line;
at each service drop location $j$ along the line, there are several loads connected in parallel, which draw a total load
current $L_j$. We will here investigate the ``value'' of a measurement of $L_j$ to the operator.
Such a measurement could be realized in a modern distribution power grid by using individual customers' smart meters.
For example, the individual customers at location $j$ could aggregate their measurements and send them
directly to the operator, in
real time. However, for privacy reasons, it may not be desirable for an individual customer
to share its instantaneous load. We denote such a customer by $C$ in Fig.~\ref{fig:powerline}, and will next investigate how the customer can contribute with measurements and still maintain a quantifiable amount of privacy.
\begin{rem}
It is not necessary to have a trusted agent at location $j$ to perform the aggregation of the measured load currents, to form
the measurement of $L_j$.
In~\cite{acscastelluccia}, it is shown how such computations can be securely decentralized among the customers.
\end{rem}

\begin{figure}[tb]
\centering
\includegraphics[width=0.75\hsize]{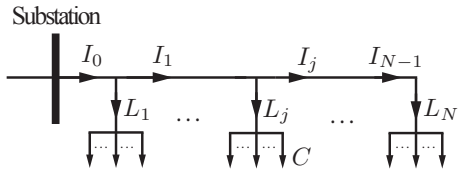}
\caption{The considered distribution line. The grid operator can only directly measure the current $I_0$ at the substation.
We analyze how the operator can benefit in the state estimation from communicated measurements of the loads $L_1,L_2,\ldots,L_N$, while maintaining individual customers' privacy (customer denoted by $C$).}
\label{fig:powerline}
\end{figure}

\subsubsection{Modeling}
In what we call the \emph{Base Scenario},
the operator has only access to a real-time measurement $Z_0$ of the total current flow $I_0$ that is being supplied from
the substation,
\begin{equation}
Z_0 = I_0 + W_0,
\label{eq:flow0}
\end{equation}
where $W_0 \sim \mathcal{N}(0,R_0)$ is zero-mean Gaussian measurement noise, of variance $R_0 \equiv \sigma_0^2$ ($\sigma_0$ is the standard deviation). In the Base Scenario, we also assume the operator has access to a statistical model of the loads $L_1,L_2,\ldots, L_N$ along the line. In the power systems literature,
sometimes such data are referred to as \emph{pseudo measurements}.
The load vector $L$ is assumed to have a multivariate Gaussian distribution, $L \sim \mathcal{N}(m,P)$, with mean and covariance denoted by
\begin{align*}
m & = \begin{pmatrix} m_1 \\ \vdots \\ m_N \end{pmatrix} := \E L  = \begin{pmatrix} \E {L_1} \\ \vdots \\ \E {L_N} \end{pmatrix}  \\
P & = \begin{pmatrix}  P_{11} & \ldots &  P_{1N} \\ \vdots & \ddots & \vdots  \\ P_{N1} & \ldots & P_{NN}\end{pmatrix} := \E {(L-m)(L-m)^T}.
\end{align*}
The Gaussian distribution of the loads can be justified if we assume each load $L_j$ is an aggregate of a multitude of individual customers as indicated already in Fig.~\ref{fig:powerline}. The operator could estimate this statistical model based on
historical load data. The following notation will also be used,
\begin{equation*}
\begin{pmatrix}  P_{1} \\\vdots \\ P_{N} \end{pmatrix} := \begin{pmatrix}  P_{11}+\ldots + P_{1N} \\\vdots \\ P_{N1}+\ldots + P_{NN} \end{pmatrix}  = P\mathbf{1},
\end{equation*}
where $\mathbf{1}$ is column vector of ones.

In this paper, the sole physical model used is based on the \emph{conservation of currents}, \ie
\begin{equation}
I_j = \sum_{k>j} L_k, \quad j=0,\,1,\ldots,N,
\label{eq:flow}
\end{equation}
where $I_j$ are the (unknown) currents along the line. It is clear that the operator may
benefit from a more detailed physical model in the state estimation to follow, but we leave this option for future work.
\begin{rem} Note that flow conservation models similar to \eqref{eq:flow} apply to other distribution networks, such as gas and water networks, where mass is a conserved quantity. Hence, the results that follow are applicable also in these scenarios.
\end{rem}

Based on the model \eqref{eq:flow}, it follows that the total current $I_0$ has a Gaussian distribution, $I_0 \sim \mathcal{N}(m_0,P_0+R_0)$, where
\begin{equation*}
m_0 := m_1 + m_2 + \ldots + m_N, \quad P_0 := \mathbf{1}^T P \mathbf{1}.
\end{equation*}

In what we call the \emph{Smart Meter Scenario}, the operator is also provided real-time measurements $Z_j$ of the loads $L_1,L_2,\ldots,L_N$,
\begin{equation}
Z_j = L_j + W_j, \quad j=1,2,\ldots,N,
\label{eq:smartmeas}
\end{equation}
where $W_j \sim \text{Lap}(b_j)$ are independent zero-mean random noise variables with Laplacian distributions of variance $R_j=2b_j^2$.
(The probability density function (PDF) of a zero-mean Laplacian random variable $W$ is $p_W(w)=\frac{1}{2b}e^{-|w|/b}$.)
The reason for the Laplacian noise assumption will become clear in Section~\ref{sec:diffpriv}.

\subsubsection{Problem Formulation}
Based on the measurement $Z_0$ and the statistical model of $L$, it is possible for the operator in the Base Scenario to estimate the loads $L_1,L_2,\ldots,L_{N}$ and currents $I_1,I_2,\ldots,I_{N-1}$ along the line, see Section~\ref{sec:statest}. With such estimates, one can monitor the status of the line and detect faults in real time.
However, when the load uncertainty is large (the covariance $P$ is large), the quality of this estimate is low.

The main problem considered in this paper is how to integrate the load measurements $Z_1,Z_2,\ldots,Z_N$ in the state estimation in the Smart Meter Scenario, and to quantify the improvement of the estimate and the \emph{loss of privacy} that the customers experience.

\section{Differential Privacy for the Distribution Line Loads}
\label{sec:diffpriv}
In this section, the goal is to quantify the loss of privacy that the customers along the distribution line
experience through the introduction of the load measurements $Z_0,Z_1,\ldots,Z_N$.

\subsection{Differential Privacy}
To quantify the level of privacy enjoyed by the customers along the line, we employ the notion of
\emph{differential privacy}, see~\cite{dwork+06,dwork+10,nypappas,Huang+14}, for example.
We first need to introduce some concepts:
Suppose that data from $m$ customers are stored in a vector $d\in\mathbb{R}^m$.
We call the vector $d$ a \emph{data vector}. We say two data vectors $d,d'\in\mathbb{R}^m$
are \emph{adjacent}, $\text{Adj}(d,d')$, if and only if they differ in one entry:
\begin{align*}
\text{for some } k,\,d_{k}\neq d'_{k}, \text{ and } d_{l}=d'_{l} \text{ for all } l\neq k.
\end{align*}
A concrete example is that one customer turns off his smart meter and does not participate in the data collection, transforming $d$ into $d'$. We also note that Adj$(\cdot,\cdot)$ is a symmetric binary relation.

 We say a measurement $M(\cdot,W):\mathbb{R}^m\to \mathbb{R}$, where $W$ is a
 random variable, is $(\epsilon,\delta)$-\emph{differentially private} if for all
  adjacent data vectors $d,d'\in\mathbb{R}^m$, and possible events
  $E \subseteq \text{Supp}(M(d,W))\cup \text{Supp}(M(d',W))$, it holds
\begin{equation}
\text{Pr}[M(d,W)\in E] \leq e^\epsilon \, \text{Pr}[M(d',W)\in E]+\delta.
\label{eq:diffpriv}
\end{equation}
If $\epsilon \geq 0$ and $\delta \geq 0$ are small, this means that the changes in the statistics of the output of $M$ are very small when applied to any adjacent data vectors. Put differently, with access to only the output of $M$, it is practically impossible to decide whether a single customer has participated in the measurement, or not, and $(\epsilon,\delta)$ provides a quantitative certificate of privacy to the customer.
If \eqref{eq:diffpriv} holds with $\delta=0$, we say $M$ is $\epsilon$-\emph{differentially private}.
More details on the interpretation of differential privacy can be found in \cite{dwork+06,dwork+10,nypappas,Huang+14}, for example.

To design $(\epsilon,\delta)$-differentially private measurements $M$, we use Theorems~2 and 3 in \cite{nypappas}. First, we
suppose the measurements can be written in the general form
\begin{equation*}
M(d,W) = q(d) + W,
\end{equation*}
where $q:\mathbb{R}^m \to \mathbb{R}$ is a deterministic \emph{query} that depends on the data vector only.
Next, define the \emph{sensitivity} of a query $q$ by
\begin{equation*}
S(q) := \sup_{d,d':\,\text{Adj}(d,d')} |q(d)-q(d')|.
\end{equation*}
That is, $S(q)$ measures the worst-case change in output of $q(d)$ over adjacent data vectors $d$.
The following proposition then holds.
\begin{prop}[\cite{nypappas}]
The measurement $M(d,W) = q(d) + W$ is:
\begin{itemize}
\item[(a)] $\epsilon$-differentially private if the random variable $W$ has a Laplace distribution, $W \sim \text{Lap}(b)$, with zero mean and scale parameter
$b \geq \frac{S (q)}{\epsilon}$;
\item[(b)] $(\epsilon,\delta)$-differentially private if the random variable $W$ has a Gaussian distribution, $W \sim \mathcal{N}(0,\sigma^2)$, with zero mean and standard deviation
$\sigma \geq \frac{S (q)}{2\epsilon}(K + \sqrt{K^2+2\epsilon})$,
where $K:=K(\delta)=Q^{-1}(\delta)$ and $Q(x):=\frac{1}{\sqrt{2\pi}}\int_x^\infty e^{-u^2/2}du$.
\end{itemize}
\label{prop:LapGauss}
\end{prop}

Hence, by either adding Laplacian noise or Gaussian noise to a query we obtain differential privacy.
Note that Laplacian noise yields better privacy ($\delta=0$) because its PDF has fatter tails. Thus
if customers can choose to add noise, Laplacian noise is the natural selection and justifies the model~\eqref{eq:smartmeas}.
The noise in \eqref{eq:flow0} is modeled as being Gaussian, because we assume it arises in the physical meter
owned by the operator. There is of course also similar physical noise in the customers' smart meters, but for simplicity we assume this is negligible.
In any case, introducing more noise sources in the customers' measurements would only increase their privacy.

\subsection{Privacy in the Base Scenario}
By providing measurements of the loads $L_j$, it is clear that the customers will loose some privacy.
However, the operator can already measure the total current $I_0$ by means of \eqref{eq:flow0}, beyond the control of the customers. We should first therefore evaluate the loss of privacy through this measurement. Because $Z_0$ is subjected to Gaussian noise, we use Proposition~\ref{prop:LapGauss}-(b).

The query in this case is the total current, $q_0 = l_1 + l_2 +\ldots + l_N$, where $l_1,l_2,\ldots,l_N$ are realizations of the random variables $L_1,L_2,\ldots,L_N$. An adjacent load pattern is obtained if
a single customer $C$ at an arbitrary location $j$ changes his or her load, so that the load goes from $l_j$ to $l_j'$. If we assume
a (uniform) bound on the difference between the maximum and minimum load current of a single customer is $\Delta$, it easily follows that the sensitivity of $q_0$ is
\begin{equation*}
S (q_0) = \Delta.
\end{equation*}
The differential privacy in the Base Scenario is stated next.
\begin{lem}
\label{lem:eps0}
The measurement $Z_0$ gives $(\epsilon_0,\delta_0)$-differential privacy to the customers, where
\begin{equation*}
\epsilon_0 = \frac{\Delta K}{\sigma_0} + O\left(\frac{\epsilon_0}{K}\right), \quad \epsilon_0 \rightarrow 0,
\end{equation*}
and $K=K(\delta_0)=Q^{-1}(\delta_0)$.
\end{lem}
\begin{proof}
Follows by a Taylor expansion of the standard deviation formula in Proposition~\ref{prop:LapGauss}-(b).
\end{proof}

A typical choice of $\delta_0$ is $0.05$, which gives $K\approx 1.64$, or $\delta_0 = 0.01$, which gives $K\approx 2.33$. We note that the $O(\cdot)$-term in Lemma~\ref{lem:eps0} can be neglected if the customer size $\Delta$ is small as compared to $K$, and then $\epsilon_0 \approx \Delta K/\sigma_0$.

\begin{rem}
Note that differential privacy is independent of the underlying probability distribution of the load pattern $L$ (the data). Hence, no matter if it is really Gaussian, or not, the level of privacy is guaranteed. The statistical model of the data is needed for reliable state estimation, however.
\end{rem}

\subsection{Privacy in the Smart Meter Scenario}
In the Smart Meter Scenario,
the customers provide the measurements $Z_j$, $j=1,2,\ldots,N$ to the operator,
which causes a loss of privacy compared to the Base Scenario.
In this case, the queries are of the type $q_j=l_j$, where $l_j$ is a realization of $L_j$.
Under the same assumptions as in the previous subsection, we have that
\begin{equation*}
S (q_j) = \Delta, \quad j=1,2,\ldots,N,
\end{equation*}
where $\Delta$ is the uniform bound on a single customer's maximum load current change.
\begin{rem}
Note that the measurement $Z_j$ sent to the operator is the sum (aggregate) of all the customers' load currents at location $j$, plus the chosen Laplacian noise. We assume that $\Delta$ is a bound that works for every single customer at location~$j$.
\end{rem}

By Proposition~\ref{prop:LapGauss}-(a), it is clear that the measurement $Z_j$ gives
$\epsilon$-differential privacy to every single customer, where $\epsilon = \Delta/b_j$. Hence, by choosing the scaling parameter $b_j$ for the Laplacian noise properly, the customers can tune the level of privacy. However, there is also a loss of privacy through $Z_0$, and for the composite measurement $(Z_0,Z_j)$ we have the following result.
\begin{thm}
\label{thm:compositeprivacy}
The composite measurement $(Z_0,Z_j)$ gives a customer at location $j$ a $(\epsilon_0+\epsilon, \delta_0e^{\epsilon})$-differential privacy, when $Z_0$ is
$(\epsilon_0,\delta_0)$-differentially private and $Z_j$ is $\epsilon$-differentially private.
\end{thm}
\begin{proof}
By definition of differential privacy, we have that
\begin{align*}
\text{Pr}[Z_0\in E_0|L_j=l_j] &\leq e^{\epsilon_0} \, \text{Pr}[Z_0\in E_0|L_j=l_j']+\delta_0 \\
\text{Pr}[Z_j\in E_j|L_j=l_j] &\leq e^\epsilon \, \text{Pr}[Z_j\in E_j|L_j=l_j'],
\end{align*}
for all events $E_0$ and $E_j$ and adjacent loads $|l_j-l_j'|\leq \Delta$. Multiplying the two inequalities, and using that the measurements are independent, we obtain
\begin{multline*}
\text{Pr}[Z_0\in E_0,Z_j\in E_j|L_j=l_j] \\ \leq e^{\epsilon_0+\epsilon} \, \text{Pr}[Z_0\in E_0,Z_j\in E_j|L_j=l_j']+\delta_0 e^{\epsilon},
\end{multline*}
which yields the result.
\end{proof}

That there is not a dramatic loss of privacy under compositions of differentially private measurements is in fact one of the
main advantages of differential privacy, see \cite{dwork+10} for a more general treatment.

\section{State Estimation of the Distribution Line}
\label{sec:statest}
In this section, we present algorithms for computing (optimal) state estimates, and we bound their accuracies.
To simplify the presentation, we here only consider the estimates of an arbitrary load $L_j$, and not currents $I_j$.
Because of the linear relation \eqref{eq:flow}, it should be clear how to obtain estimates of the currents from the load estimates.

\subsection{Optimal State Estimation in the Base Scenario}
In the Base Scenario, the operator has only access to the measurement $Z_0$ and the statistical model of the loads $L$.
Since these have a multivariate Gaussian distribution, the optimal state estimate (minimum mean square error (MMSE) estimate) is
given by the conditional mean of $L_j$ \cite[Example~3.4]{andersonmoore},
\begin{equation}
\hat L_j^0 := \E{L_j|Z_0} = m_j + \frac{P_j}{P_0+R_0}(Z_0-m_0),
\label{eq:hatLj0}
\end{equation}
where $P_j = P_{j1}+\ldots + P_{jj} + \ldots + P_{jN}$.
The variance of the estimation error is
\begin{equation}
Q_j^0 := \E{(\hat L_j^0 - L_j)^2} = P_{jj} - \frac{P_j^2}{P_0 + R_0},
\label{eq:Qj0}
\end{equation}
and provides a baseline accuracy indicator for state estimation.
It is clear that if the variance of the load $L_j$ ($P_{jj}$) is small in comparison to the variance of the measurement $Z_0$,
the operator gains very little information through this measurement.

\subsection{MAP State Estimation with Smart Meters}
When the load model and the measurements all have a jointly Gaussian distribution, the MMSE estimate and the
maximum a~posteriori probability (MAP) estimate (see \cite{Saha+13}, for example) coincide, and are equal to $\E{L_j|Z_0}$. In the Smart Meter Scenario,
the random noise $W_1,W_2,\ldots,W_N$ follow a Laplacian distribution, and therefore this is no longer the case, in general. It is well known \cite[Theorem~3.1]{andersonmoore}
that the MMSE estimate is still given by the conditional mean $\hat L_j^\star := \E{L_j|Z_0,Z_1,\ldots,Z_N}$,
but a simple closed-form expression for $\hat L_j^\star$ similar to
\eqref{eq:hatLj0} is in this case not known to us. The MAP estimate also does not have a simple closed-form solution, but it is
possible to compute it online by means of a simple convex program, as we show next.

The PDF for the load vector $L$ takes the form
\begin{equation*}
p_L(l) = \frac{(2\pi)^{-N/2}}{\sqrt{\det P}} e^{- \|P^{-1/2}(l-m)\|_2^2/2},
\end{equation*}
and the conditional PDF for the measurements are
\begin{equation*}
p_{Z_0,Z|L}(z_0,z|l) = \frac{1}{\sqrt{2\pi}\sigma_0(2b)^N} e^{-(z_0-l_0)^2/(2\sigma_0^2)} e^{-\|z-l\|_1/b},
\end{equation*}
where $Z:=(Z_1,Z_2,\ldots,Z_N)^T$ and $l_0:=\mathbf{1}^T l$.
Given the received measurements $z_0$ and $z$, the MAP estimate of $L_j$ can be obtained as
\begin{equation}
\hat L^{\text{MAP}}(z_0,z) := \arg \max_l p_{\bar Z|L}(z_0,z|l)p_L(l).
\label{eq:LMAP}
\end{equation}
The following proposition states an alternative, and more explicit form of the estimate.
\begin{prop}
\label{prop:MAP}
The MAP estimate of $L$ given the measurements $z_0,z$ is
\begin{multline*}
\hat L^{\text{MAP}}(z_0,z) = \arg \min_l \frac{(z_0 - \mathbf{1}^T l )^2}{2\sigma_0^2} \\+
\frac{1}{2}\|P^{-1/2}(l-m)\|_2^2   + \frac{1}{b}\|z-l\|_1.
\end{multline*}
\end{prop}

\vspace{0.3cm}
\begin{proof}
Follows by a straightforward insertion of the PDFs into \eqref{eq:LMAP}.
\end{proof}

The optimization problem in Proposition~\ref{prop:MAP} is convex since the objective is a sum of $1$-norm and $2$-norm expressions. Hence, it can easily be solved using standard optimization software. Although this result is useful
for an implementation of the state estimator, we are not aware of an accuracy indicator similar to $Q_j^0$. Because of this, we next turn to an optimal linear estimator, where we can characterize the accuracy in general.

\begin{rem}
Recently \cite{Carrillo+15} investigated dynamical state estimation problems
where the noise was a mix of Gaussian and Laplacian random
variables. Although the motivation for that work was not related to privacy, it seems the results obtained there
have implications for privacy as well.
\end{rem}

\subsection{LMMSE State Estimation with Smart Meters}
\label{sec:LMMSE}
For arbitrarily distributed measurements, it turns out that if we limit ourselves to only \emph{linear} estimators when minimizing the mean square error (the LMMSE), we obtain simple closed-form expressions. We will use the following
result.
\begin{prop}[Theorem~2.1 \cite{andersonmoore}]
\label{prop:andmoore}
Let the random variable $(X,Z)$ have mean $(m_x,m_z)$ and covariance $\bigl[ \begin{smallmatrix}\Sigma_{xx} &
\Sigma_{xz} \\ \Sigma_{yz} & \Sigma_{zz} \end{smallmatrix} \bigr]$. Then the LMMSE estimate is given by
\begin{equation*}
\Estar{X|Z} := m_x + \Sigma_{xz}\Sigma_{zz}^{-1}(Z-m_z),
\end{equation*}
with estimation error covariance $\Sigma_{xx}-\Sigma_{xz}\Sigma_{zz}^{-1}\Sigma_{zx}$.
\end{prop}

In the Smart Meter Scenario, using Proposition~\ref{prop:andmoore} one easily obtains the LMMSE estimate of $L_j$ given all the measurements as
\begin{equation*}
\begin{aligned}
\hat L_j^{0:N}&:=\Estar{L_j|Z_0,Z}  \\
& = m_j + \begin{bmatrix} P_j \\ P_{:j} \end{bmatrix}^T
\begin{bmatrix} P_0 + R_0 & \mathbf{1}^TP \\ P\mathbf{1} & P + R \end{bmatrix}^{-1}
\begin{bmatrix} Z_0 - m_0 \\ Z-m \end{bmatrix},
\end{aligned}
\end{equation*}
where $P_{:j}$ is the $j$-th column of the covariance matrix $P$. The accuracy of the LMMSE estimate is
\begin{equation}
\begin{aligned}
Q^{0:N}_j & := \E{(L_j - \hat L_j^{0:N})^2} \\ & = P_{jj} - \begin{bmatrix} P_j \\ P_{:j} \end{bmatrix}^T \begin{bmatrix} P_0 + R_0 & \mathbf{1}^TP \\ P\mathbf{1} & P + R \end{bmatrix}^{-1}\begin{bmatrix} P_j \\ P_{:j} \end{bmatrix}.
\end{aligned}
\label{eq:Qlin0N}
\end{equation}
It should be clear that the accuracy is not as good as the MMSE estimate $\hat L_j^\star$, in general, since we
have only optimized over the linear estimators.

To get an even simpler closed-form expression, an LMMSE estimate based on only
a few of the measurements is constructed.
\begin{prop}
\label{prop:LMMSE}
The LMMSE estimate of $L_j$ given the measurements $Z_0,Z_j$ is
\begin{equation*}
\begin{aligned}
\hat L_j^{0,j}&:=\Estar{L_j|Z_0,Z_j}  \\
& =\hat L_j^0 + K_j \left[(Z_j-m_j)  - \frac{P_j}{R_0+P_0}(Z_0-m_0)\right],
\end{aligned}
\end{equation*}
where $\hat L_j^0$ is the Base Scenario MMSE~estimate~\eqref{eq:hatLj0}, and
\begin{equation}
K_j = \frac{(R_0+P_0)P_{jj} - P_j^2}{(R_0+P_0)(P_{jj}+R_j)-P_j^2} \in [0,1].
\label{eq:Kj}
\end{equation}
The accuracy of $\hat L_j^{0,j}$ is
\begin{equation}
Q^{0,j}_j := \E{(L_j - \hat L_j^{0,j})^2} = Q_j^0(1 - K_j)\leq Q_j^0,
\label{eq:Qlin0j}
\end{equation}
where $Q_j^0$ is the accuracy of the Base Scenario MMSE~estimate~\eqref{eq:Qj0}.
\end{prop}
\begin{proof}
Applying the Schur complement formula to the inverse covariance matrix in the LMMSE~estimate
\begin{equation*}
\begin{aligned}
\hat L_j^{0,j}&:=\Estar{L_j|Z_0,Z_j}  \\
& = m_j + \begin{bmatrix} P_j \\ P_{jj} \end{bmatrix}^T
\begin{bmatrix} P_0 + R_0 & P_j \\ P_j & P_{jj} + R_j \end{bmatrix}^{-1}
\begin{bmatrix} Z_0 - m_0 \\ Z_j-m_j \end{bmatrix},
\end{aligned}
\end{equation*}
yields the desired expression. After some manipulations of \eqref{eq:Qlin0N} applied to only the measurements
$Z_0$ and $Z_j$, the expression in \eqref{eq:Qlin0j} is obtained.
\end{proof}

It is seen from Proposition~\ref{prop:LMMSE} that the variable $K_j$ can both be seen as the gain one should use to fuse in the additional measurement $Z_j$ to the estimate of $L_j$, and
as a measure of the relative improvement in the estimate as compared to the Base Scenario, \ie $K_j = \frac{Q_j^0-Q^{0,j}_j}{Q_j^0}$.

The relation between the accuracies of the different estimators is summarized next.
\begin{cor}
For $j=1,2,\ldots,N$, it holds that
\begin{equation*}
Q^\star_j \leq Q_{j}^{0:N} \leq Q_j^{0,j} \leq Q_j^{0},
\end{equation*}
where $Q^\star_j := \E{(L_j - \hat L_j^\star)^2}$ and $\hat L_j^\star$ is the MMSE estimate given measurements $Z_0,Z$.
\end{cor}

The reason we are interested in $\hat L_j^{0,j}$ and $Q_j^{0,j}$ is because they are easily computed and provide
a simple non-trivial expression for the improvement given by a single load measurement $Z_j$. The point is that if a more complicated estimate is constructed, such as  $\hat L_j^\star$ or $\hat L_j^{0:N}$, their accuracy will be at least as good as indicated by $K_j$ in \eqref{eq:Qlin0j}. This will prove useful in the next section where some insights related to the trade-off between privacy and estimation quality are obtained.

\section{Trade-off: Estimation Performance vs. Privacy Loss}
\label{sec:tradeoff}
A customer in the Smart Meter Scenario has $(\epsilon+\epsilon_0,\delta_0 e^\epsilon)$-differential privacy, following Theorem~\ref{thm:compositeprivacy}.
The customer can increase his or her privacy by increasing the smart meter noise level $R_j$, but this comes at the expense of the state estimation performance.
In this section, we investigate this trade-off a bit closer, and in particular identify situations where only a minor loss of privacy can lead to a large improvement in the
state estimate.

We start by introducing the dimensionless quantities
\begin{equation*}
\eta_j := \frac{\Delta^2}{P_{jj}}, \quad \zeta_j := \frac{P_{jj}}{P_0+R_0}.
\end{equation*}
The variable $\eta_j$ normalizes the customer's variability with the total load variability at location $j$.
Hence, it is a measure of how large player the customer is in comparison with other loads at location $j$. The variable $\zeta_j$ measures how
large the total load variability at location $j$ is in comparison with the overall current variability in the distribution line. Thus if both $\eta_j$ and $\zeta_j$ are small, it means the customer is a relatively small player at location $j$, and location $j$ is overall responsible for only a small part of the variability on the line.

In the following, we for simplicity assume that the loads are uncorrelated, \ie that $P_j = P_{jj}$. We use that
the variance of the Laplacian measurement noise added to $L_j$ is
\begin{equation*}
R_j = \frac{2\Delta^2}{\epsilon^2} = \frac{2\eta_j P_{jj}}{\epsilon^2}.
\end{equation*}
We can now express the relative improvement of the state estimate as (see \eqref{eq:Kj}--\eqref{eq:Qlin0j})
\begin{equation}
\label{eq:tradeoff}
\begin{aligned}
K_j & = \frac{(R_0+P_0)P_{jj} - P_j^2}{(R_0+P_0)(P_{jj}+R_j)-P_j^2} \\
& = \frac{1}{1+\frac{R_j(R_0+P_0)}{(R_0+P_0)P_{jj} - P_{jj}^2}} \\
& = \frac{1}{1 + \frac{2\eta_j}{\epsilon^2  (1 - \zeta_j)}} \approx \frac{\epsilon^2(1 - \zeta_j)}{2\eta_j},
\end{aligned}
\end{equation}
where the last approximation holds for small $\epsilon$. The total differential
privacy for the customer is $(\epsilon+\epsilon_0,\delta_0e^\epsilon)$ and $\epsilon$ is a measure of
how much extra privacy the customer at location $j$ gives up by sharing his or her load measurement. The customer's differential
privacy can be no better than $(\epsilon_0,\delta_0)$, and using Lemma~\ref{lem:eps0} this lower bound is
\begin{equation}
\epsilon_0^2 \approx \frac{\Delta^2K^2}{R_0} = \eta_j \zeta_j K^2\left(1+\frac{P_0}{R_0} \right).
\label{eq:eps0approx}
\end{equation}
Not surprisingly, small $\eta_j$ and $\zeta_j$ assure a high level of privacy to the customer in the Base Scenario because his or her actions are barely noticeable in $Z_0$.

\begin{figure}[tb]
\centering
\includegraphics[width=0.9\hsize]{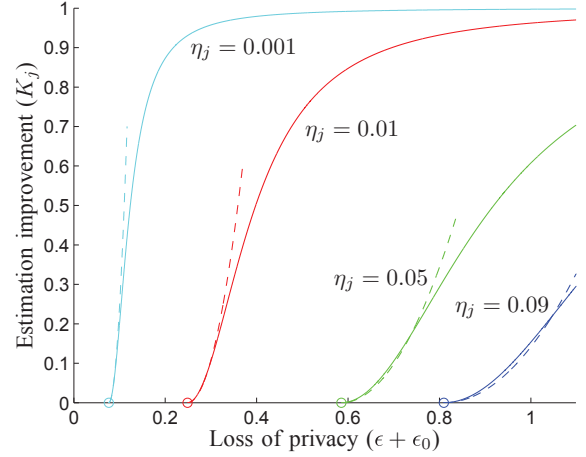}
\caption{Lower bound on the relative improvement of the state estimate in the Smart Meter Scenario ($K_j$), as a function of total loss of customer privacy, when $P_0=1$, $R_0=0.05$, $\delta_0=0.05$, $\zeta_j=0.1$, and varying customer variability $\eta_j$.
It is seen that the curvature dramatically increases for customers who have a high level of privacy in the Base Scenario (small $\epsilon_0$).}
\label{fig:etaplot}
\end{figure}

\begin{figure}[tb]
\centering
\includegraphics[width=0.9\hsize]{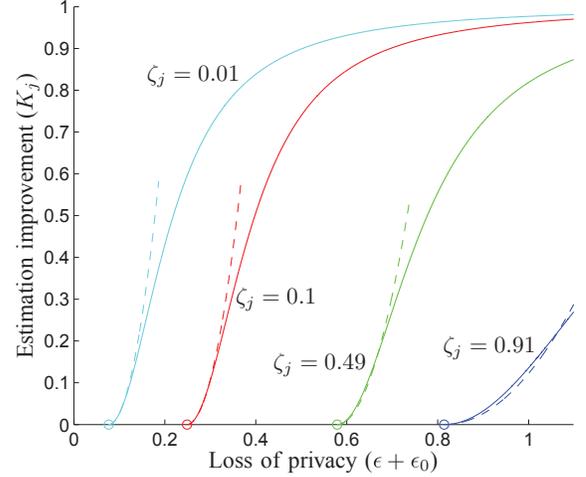}
\caption{Lower bound on the relative improvement of the state estimate in the Smart Meter Scenario ($K_j$), as a function of total loss of customer privacy, when $P_0=1$, $R_0=0.05$, $\delta_0=0.05$, $\eta_j=0.01$, and varying load variability $\zeta_j$.
It is seen that the curvature increases for customers who have a higher level of privacy in the Base Scenario
(small $\epsilon_0$).}
\label{fig:zetaplot}
\end{figure}

\addtolength{\textheight}{-8.7cm}   

In Figs.~\ref{fig:etaplot} and \ref{fig:zetaplot}, we plot the lower bound on relative improvement \eqref{eq:tradeoff} as a function of the
total loss of privacy, for fixed $\zeta_j$ and $\eta_j$, respectively. The values of $\eta_j$ and $\zeta_j$ are
chosen to give comparable values for the lower bounds $\epsilon_0$ (indicated by circles in Figs.~\ref{fig:etaplot} and \ref{fig:zetaplot}). It is interesting to see that
the improvement curves increase faster for small values of $\epsilon_0$, especially as
the customer's variability $\eta_j$ decreases. This is also clearly seen in the quadratic approximation \eqref{eq:tradeoff}, whose curvature
is $(1-\zeta_j)/(2\eta_j)$. Combining \eqref{eq:tradeoff} and \eqref{eq:eps0approx},
\begin{equation}
K_j \approx \frac{\epsilon^2}{2 \eta_j}(1-\zeta_j) \approx \frac{K^2}{2}\left(1+\frac{P_0}{R_0}\right)\zeta_j (1-\zeta_j) \left(\frac{\epsilon}{\epsilon_0}\right)^2
\end{equation}
we also directly see that a small $\epsilon_0$ increases the curvature. (The quadratic approximations are indicated by
dashed lines in Figs.~\ref{fig:etaplot} and \ref{fig:zetaplot}.) The importance of this observation is that
a customer who enjoys a high level of differential privacy in the Base Scenario can dramatically improve the
state estimation performance by only giving up a little bit extra of privacy. For example, in Fig.~\ref{fig:etaplot}, we see
that if a customer is willing to increase $\epsilon_0+\epsilon$ from $0.25$ to $0.35$ when $\eta_j=0.01$,
there is at least around $30\%$ improvement in estimation quality. On the other hand, the
smart meter measurements of a customer who already has a low level of privacy in the Base Scenario are of much lower value.
The reason is that such customers are already possible to estimate quite accurately using the measurement $Z_0$ alone.

\section{Conclusion}
\label{sec:conclusion}
We have proposed a modeling framework to analyze the differential privacy for customers equipped with smart meters in a radial distribution power grid. Several different state estimation schemes were proposed, and an interesting trade-off between
the utility for the operator and the loss of privacy for the customer was identified. The analysis revealed that
aggregated measurements from small customers can significantly improve the operator's state estimate at a small loss of privacy.

It should be noted that in the current framework it is not possible for the customer to influence the
baseline $(\epsilon_0,\delta_0)$-differential privacy. This would be possible if the customer has access to a local controllable energy storage, see \cite{Tan+13}, for example. Open problems for future research also include more detailed modeling of the distribution grid and quantification of the accuracy of the MAP~estimate.

\bibliographystyle{IEEEtran}

\end{document}